\documentclass[12pt,reqno]{amsart}
\usepackage{amsmath, amsfonts, amssymb, amsthm}
\def\Z{\mathbb Z}

\def\Q{\mathbb Q}
\def\R{\mathbb R}

\def\F{\mathbb F}
\def\cR{\mathcal R}
\def\cH{\mathcal H}
\def\cQ{\mathcal Q}

\def\fp{\mathfrak p}
\def\bi{\mathbf i}

\def\1{{\bf 1}}

\def\jacob #1#2{\genfrac{(}{)}{}{}{#1}{#2}}
\def\pmod #1{\ ({\rm{mod}}\ #1)}

\theoremstyle{plain}
\newtheorem{theorem}{Theorem}

\newtheorem{lemma}{Lemma}

\theoremstyle{definition}
\newtheorem*{acknowledgment}{Acknowledgments}
\theoremstyle{remark}

\begin{document}

\title{Some permutations over $\F_p$ concerning primitive roots}

\begin{abstract}
Let $p$ be an odd prime and let $\F_p$ denote the finite field with $p$ elements. Suppose that $g$ is a primitive root of $\F_p$. Define the permutation $\tau_g:\,\cH_p\to\cH_p$ by
$$
\tau_g(b):=\begin{cases} g^b,&\text{if }g^b\in\cH_p,\\
-g^b,&\text{if }g^b\not\in\cH_p,\\
\end{cases}
$$
for each $b\in\cH_p$, where $\cH_p=\{1,2,\ldots,(p-1)/2\}$ is viewed as a subset of $\F_p$. In this paper, we investigate the sign of $\tau_g$. For example, if $p\equiv 5\pmod{8}$, then
$$
(-1)^{|\tau_g|}=(-1)^{\frac{1}{4}(h(-4p)+2)}
$$
for every primitive root $g$, where $h(-4p)$ is the class number of the imaginary  quadratic field $\Q(\sqrt{-4p})$.
\end{abstract}

\author{Li-Yuan Wang}
\address{Department of Mathematics, Nanjing University, Nanjing 210093,
People's Republic of China}
\email{wly@smail.nju.edu.cn}
\author{Hao Pan}
\address{School of Applied Mathematics, Nanjing University of Finance and Economics, Nanjing 210046, People's Republic of China}
\email{haopan79@zoho.com}
\keywords{permutation; primitive root of a prime; }

\subjclass[2010]{Primary 11T22; Secondary 	05A05, 11R29, 12E20 }

\thanks{The first author is supported by  the National Natural Science Foundation of China (Grant No. 11571162). The second author is supported by the National Natural Science Foundation of China (Grant No. 11671197).}

\maketitle

\section{Introduction}
\setcounter{lemma}{0}
\setcounter{theorem}{0}
\setcounter{corollary}{0}
\setcounter{remark}{0}
\setcounter{equation}{0}
\setcounter{conjecture}{0}
Suppose that $p$ is an odd prime. Let $\F_p$ denote the finite field with $p$ elements and let $\F_p^\times=\F_p\setminus\{0\}$. For convenience, we may identify $\F_p$ with $\{0,1,\ldots,p-1\}$.  
Suppose that  $g$ is a primitive root of $\F_p$. Define 
\begin{equation}
\sigma_g(b):=g^b
\end{equation}
for each $b\in\{1,\ldots,p-1\}$.
Since $\F_p^\times$ is identified with $\{1,\ldots,p-1\}$, we can view $\sigma_g$ as a permutation over $\F_p^\times$. In \cite{KLP}, Kohl considered the sign of the permutation $\sigma_g$ and proposed an interesting problem:

\begin{itemize} 
\item \textit{If $p\equiv 1\pmod{4}$, then
\begin{equation}\label{sigmap14}
|\{g\in\cR_p:\,(-1)^{|\sigma_g|}=1\}|=|\{g\in\cR_p:\,(-1)^{|\sigma_g|}=-1\}|,
\end{equation}
where $(-1)^{|\sigma_g|}$ denotes the sign of $\sigma_g$ and
$$\cR_p:=\{g\in\F_p:\,g\text{ is a primitive root}\}.
$$}

\item \textit{If $p\equiv 3\pmod{4}$, then
\begin{equation}\label{sigmap34}
(-1)^{|\sigma_g|}\equiv-\bigg(\frac{p-1}{2}\bigg)!\pmod{p}
\end{equation}
for each $g\in\cR_p$.}
\end{itemize}

\medskip\noindent Soon, (\ref{sigmap14}) and (\ref{sigmap34}) were confirmed
by Ladisch and Petrov \cite{KLP} respectively. In particular, the key ingredient of Petrov's proof is the formula
\begin{equation}\label{petrov}
\prod_{\substack{1\leq i<j\leq p-1}}(\zeta^j-\zeta^i)=e^{\frac{(p-2)(3p-1)}{4}\cdot\pi\bi}\cdot(p-1)^{\frac{p-1}{2}},
\end{equation}
where $\bi=\sqrt{-1}$
and
$\zeta=e^{\frac{2\pi\bi}{p-1}}$ is the $(p-1)$-th primitive root of unity.
 A classical result of Mordell \cite{M61} says that for any prime $p\equiv 3\pmod{4}$
$$
\bigg(\frac{p-1}{2}\bigg)!\equiv(-1)^{\frac{1}{2}(h(-p)+1)}\pmod{p},
$$
where $h(d)$ denotes the class number of the quadratic field $\Q(\sqrt{d})$. Hence
(\ref{sigmap34}) also can be rewritten as
\begin{equation}\label{sigmap34b}
(-1)^{|\sigma_g|}\equiv(-1)^{\frac{1}{2}(h(-p)-1)}\pmod{p}.
\end{equation}

In this note, we shall consider a variant of Kohl's problem. Let $$
\cH_p:=\bigg\{1,2,\ldots,\frac{p-1}2\bigg\}$$ 
and view $\cH_p$ as a subset of $\F_p$. 
We shall define the permutation $\tau_g$ over $\cH_p$
for every primitive root $g\in\cR_p$. Define
\begin{equation}
\tau_g(b):=\begin{cases} g^b,&\text{if }g^b\in\cH_p,\\
-g^b,&\text{if }g^b\not\in\cH_p,\\
\end{cases}
\end{equation}
for each $b\in\cH_p$. Note that $-g^{b}=g^{\frac{1}{2}(p-1)+b}$. It is easy to see that $\tau_g$ is a permutation over $\cH_p$. 
In fact, our definition of $\tau_g$ is motivated by the well-known Gauss lemma, which says that for each $a\in\F_p^\times$, the Legendre symbol
$$
\jacob{a}p=(-1)^{N_{a,p}},
$$
where
$$
N_{a,p}:=|\{b\in\cH_p:\,ab\not\in\cH_p\}|.
$$

It is natural to ask what the sign of $\tau_g$ is. We have
\begin{theorem}\label{main}
	(i) If $p\equiv 1 \pmod 8$, then for each $g\in \cR_p$,
\begin{equation}\label{tau18}
(-1)^{|\tau_g|} =(-1)^{\frac{1}{4}h(-4p)} \cdot (-1)^{|\sigma_g|}.
\end{equation}	

 (ii) If $p\equiv 5 \pmod 8$, then for each $ g \in \cR_p$,
 \begin{equation} \label{tau58}
 (-1)^{|\tau_g|}=(-1)^{\frac{1}{4}(h(-4p)+2)}.
 \end{equation}

	(iii) If $p=18(2n+1)^2+1$ for some positive integer $n$, then 
	 \begin{equation}\label{tau182n121}
	 (-1)^{|\tau_g|}=(-1)^{n+1}.
	 \end{equation}

	(iv) If $p\equiv3\pmod{4}$ is not of the form $18(2n+1)^2+1$, then
	\begin{equation}
|\{ g\in \mathcal{R}_p  :(-1)^{ \left| \tau_g \right|}=1 \}|=|\{ g\in \mathcal{R}_p  :(-1)^{ \left| \tau_g \right|}=-1 \}|.
\end{equation}	
\end{theorem}
The proof of Theorem \ref{main} will be given in the subsequent two sections. In fact, we shall use a modification of Petrov's discussions to prove (i)-(iii) of Theorem \ref{main}. And for the final case (iv), we shall find $1\leq a\leq p-1$ with $(a,p-1)=1$ such that $\tau_{g}$ and $\tau_{g^a}$ have the opposite parity for each $g\in\cR_p$.

\section{The case $p\equiv 1\pmod{4}$}
\setcounter{lemma}{0}
\setcounter{theorem}{0}
\setcounter{corollary}{0}
\setcounter{remark}{0}
\setcounter{equation}{0}
\setcounter{conjecture}{0}

In this section, we shall prove (i) and (ii) of Theorem\ref{main}. First, suppose that $p$ is an odd prime.
Let $$\cQ_p=\{x^2:\, x\in\F_p^*\},$$ i.e., $\cQ_p$ is the set of all non-zero quadratic residues in $\F_p$. Define $\lambda:\,\cH_p\to\cQ_p$ by
$$
\lambda(b)=b^2.
$$
Since $\cH_p=\{1,\ldots,(p-1)/2\}$, clearly $\lambda$ is a bijection.
Let
$$
\nu_g:=\lambda\circ\tau_g\circ\lambda^{-1}.
$$
Then $\nu_p$ is a permutation over $\cQ_p$. Recall that the sign of a permutation is also determined by its decomposition into the product of disjoint cycles. So we must have
\begin{equation}
(-1)^{|\nu_g|}=(-1)^{|\tau_g|}.
\end{equation}

Clearly for each $1\leq b\leq (p-1)/2$, we have
$$
\nu_g(b^2)=g^{2b}.
$$
Hence
\begin{align}\label{tau_g}
(-1)^{|\nu_g|}=\prod_{1\leq i<j\leq \frac{p-1}{2}}\frac{g^{2j}-g^{2i}}{j^2-i^2}
\end{align}
over $\F_p$. 
\begin{lemma} Suppose that the prime $p\equiv1\pmod{4}$. Then
\begin{equation}\label{j2i2modp}
\prod_{1\leq i<j\leq \frac{p-1}{2}}(j^2-i^2)\equiv-\bigg(\frac{p-1}{2}\bigg)!\pmod{p}.
\end{equation}
\end{lemma}
\begin{proof}
Evidently,
\begin{align}\label{ijp12j2i2p14}
\prod_{1\leq i<j\leq \frac{p-1}{2}}(j^2-i^2)
=&\prod_{1\leq i\leq\frac{p-3}{2}}\bigg(\prod_{i<j\leq\frac{p-1}{2}}(j-i)\bigg)\cdot\prod_{1\leq i\leq\frac{p-3}{2}}\bigg(\prod_{i<j\leq\frac{p-1}{2}}(j+i)\bigg)\notag\\
=&\prod_{1\leq i\leq \frac{p-3}{2}}\bigg(\frac{p-1}{2}-i\bigg)!\cdot\prod_{1\leq i\leq \frac{p-3}{2}}\frac{(i+\frac{p-1}{2})!}{(2i)!}\notag\\
=&\prod_{\substack{1\leq k\leq p-2\\
k\neq \frac{p-1}{2}}}k!\cdot\prod_{1\leq i\leq \frac{p-3}{2}}\frac{1}{(2i)!}=\frac{1}{(\frac{p-1}{2})!}\prod_{0\leq k\leq\frac{p-3}{2}}(2k+1)!.  
\end{align}
By the classical Wilson theorem, for each $1\leq k\leq p-2$,
\begin{equation}\label{kp1kmodp}
k!(p-1-k)!\equiv (-1)^{p-1-k}k!\prod_{j=1}^{p-1-k}(p-j)=(-1)^k(p-1)!\equiv(-1)^{k+1}
\pmod{p}.
\end{equation}
Since $(p-1)/2$ is even now, we have
$$
\prod_{1\leq i<j\leq \frac{p-1}{2}}(j^2-i^2)\equiv\frac{1}{(\frac{p-1}{2})!}\equiv -\bigg(\frac{p-1}{2}\bigg)!\pmod{p}.
$$
\end{proof}
In order to evaluate $\prod_{1\leq i<j\leq \frac{p-1}{2}}(g^{2j}-g^{2i})$, we may view $g$ as an integer lying in $\{1,2,\ldots,p-1\}$. Let $\zeta=e^{\frac{2\pi\bi}{p-1}}$ be a $(p-1)$-th primitive root of unity. Clearly
$$
g^{p-1}-1=\prod_{j=1}^{p-1}(g-\zeta^j)\equiv 0\pmod{p}.
$$
Hence there exists $1\leq j_0\leq p-1$ such that $g-\zeta^{j_0}$ is not prime to $p$, i.e., 
$$
g-\zeta^{j_0}\equiv 0\pmod{\fp}
$$
for some prime ideal $\fp\subseteq\Q(\zeta)$ with $\fp\mid p$, where $\Q(\zeta)$ denotes the $(p-1)$-th cyclotomic field. For each $1\leq k\leq p-2$, since $g^k-1$ is prime to $p$ , we must have $\zeta^{kj_0}\neq 1$. So $j_0$ is prime to $p-1$, i.e., $\zeta^{j_0}$ is also a $(p-1)$-th primitive root of unity. 
Recall that $\psi:\,\zeta\mapsto\zeta^{j_0}$ gives a Galois automorphism over $\Q(\zeta)$ (cf. \cite[p. 71]{SL}).
Hence without loss of generality, we may assume that $j_0=1$, i.e.,
\begin{equation}\label{gzetafp}
g\equiv \zeta\pmod{\fp}.
\end{equation}
\begin{lemma} Suppose that $p\equiv1\pmod{4}$ is a prime. Then
\begin{equation}\label{g2jg2imodp}
\prod_{1\leq i<j\leq \frac{p-1}{2}}(g^{2j}-g^{2i})\equiv e^{\frac{(p-3)(3p+1)}{16}\cdot\pi\bi}\cdot\bigg(\frac{p-1}{2}\bigg)^{\frac{p-1}{4}}\pmod{\fp}.
\end{equation}
\end{lemma}
\begin{proof}
Let 
\begin{equation}
\Upsilon(x):=\prod_{1\leq i<j\leq \frac{p-1}{2}} (x^{2j}-x^{2i}).
\label{Gqs}
\end{equation}
Clearly
$$
\Upsilon(\zeta)^2=\prod_{1\leq i<j\leq \frac{p-1}{2}}(\zeta^{2j}-\zeta^{2i})^2 
=(-1)^{\frac{(p-1)(p-3)}{8}}\prod_{1\leq i\neq j\leq \frac{p-1}{2}}(\zeta^{2j}-\zeta^{2i}).
$$
Note that
\begin{align*}
\prod_{i=1}^{\frac{p-3}{2}}(1-\zeta^{2i})=\lim_{x\to 1}\prod_{i=1}^{\frac{p-3}{2}}(x-\zeta^{2i})=\lim_{x\to1}\frac{x^{\frac{p-1}{2}}-1}{x-1}=\frac{p-1}{2}.
\end{align*}
It follows that
\begin{align*}
\prod_{j=1}^{\frac{p-1}{2}}\prod_{\substack{1\leq i\leq\frac{p-1}{2}\\ i\neq j }}(\zeta^{2j}-\zeta^{2i})
=&\prod_{j=1}^{\frac{p-1}{2}}\zeta^{(p-3)j}\prod_{\substack{1\leq i\leq\frac{p-1}{2}\\ i\neq j }}(1-\zeta^{2(i-j)})\\
=&\prod_{j=1}^{\frac{p-1}{2}}\bigg(\zeta^{(p-3)j}\cdot\frac{p-1}{2}\bigg)
=\zeta^{\frac{(p-3)(p^2-1)}{8}}\cdot\bigg(\frac{p-1}{2}\bigg)^{\frac{p-1}{2}}.
\end{align*}
Thus we get
\begin{equation}
\Upsilon(\zeta)^2=(-1)^{\frac{(p-1)(p-3)}{8}+\frac{(p-3)(p+1)}{4}}\cdot\bigg(\frac{p-1}{2}\bigg)^{\frac{p-1}{2}}.
\end{equation}
In particular,
$$
|\Upsilon(\zeta)|=\bigg(\frac{p-1}{2}\bigg)^{\frac{p-1}{4}}.
$$

Let $\alpha=\arg \Upsilon(\zeta)$ denote the argument of $\Upsilon(\zeta)$.   Note that for any $\theta\in[0,2\pi)$,
$$
1-e^{\bi\theta}=(1-\cos\theta)-\bi\sin\theta=2\sin\frac{\theta}2\bigg(\sin\frac{\theta}{2}-\bi\cos\frac{\theta}{2}\bigg)=2\sin\frac{\theta}2\cdot e^{\bi(\frac{\theta}{2}-\frac{\pi}{2})},
$$
i.e.,
$$
\arg(1-e^{\bi\theta})\equiv\frac{\theta}{2}-\frac{\pi}{2}\pmod{2\pi},
$$
where for $x,y,z\in\R$, $x\equiv y\pmod{z}$ means $x-y=nz$ for some integer $n$. 
Hence
$$
\arg(\zeta^{2j}-\zeta^{2i})=\arg\big(-\zeta^{2i}(1-\zeta^{2j-2i})\big)\equiv
(j+i)\cdot\frac{2\pi}{p-1}+\frac\pi2\pmod{2\pi}
$$
for each $1\leq i<j\leq(p-1)/2$.
It follows that
\begin{align*}
\arg \Upsilon(\zeta)\equiv&\frac{2\pi}{p-1}\sum_{1\leq i<j\leq \frac{p-1}{2}}(j+i)+\frac{\pi}{2}\sum_{1\leq i<j\leq \frac{p-1}{2}}1\\
=&
\frac{(p-3)(p^2-1)}{16}\cdot\frac{2\pi}{p-1}+\frac{(p-3)(p-1)}{8}\cdot\frac{\pi}2\\
=&\frac{(p-3)(3p+1)}{16}\cdot\pi\pmod{2\pi}.
\end{align*}
So
\begin{equation}\label{Upsilonzeta}
\Upsilon(\zeta)=e^{\frac{(p-3)(3p+1)}{16}\cdot\pi\bi}\cdot\bigg(\frac{p-1}{2}\bigg)^{\frac{p-1}{4}}.
\end{equation}
\end{proof}
Furthermore, we also need a result of Williams and Currie (cf. \cite[p. 972]{WC82}).
\begin{lemma}
Suppose that $p\equiv 1\pmod{4}$ is a prime. Then 
\begin{equation}\label{2p14modp18}
 	2^{\frac{p-1}{4}}\equiv (-1)^{\frac{h(-p)}{4}+\frac{p-1}{8}}\pmod p
\end{equation}
if $p\equiv 1 \pmod 8$. And
\begin{equation}\label{2p14modp58}
2^{\frac{p-1}{4}}\cdot \bigg(\frac{p-1}{2}\bigg)!\equiv (-1)^{\frac{1}{4}(h(-p)+2)+\frac{p-5}{8}}\pmod p\end{equation}
 if $p\equiv 5\pmod 8$.
\end{lemma}

First, assume that case $p\equiv 5\pmod{8}$. It is easy to check that
$$
\frac{(p-3)(3p+1)}{16}\equiv \frac{p-5}{8}\pmod{2}
$$
when $p\equiv 5\pmod{8}$. So
\begin{equation}\label{Ggmodfp}
\Upsilon(g)\equiv \Upsilon(\zeta)=\bigg(\frac{p-1}{2}\bigg)^{\frac{p-1}{4}}\cdot(-1)^{\frac{p-5}{8}}\pmod{\fp}.
\end{equation}
Note that both sides of (\ref{Ggmodfp}) are rational integers. So we have
\begin{equation}
\Upsilon(g)\equiv (-1)^{\frac{p-5}{8}}\cdot\bigg(\frac{p-1}{2}\bigg)^{\frac{p-1}{4}}\equiv
\frac{(-1)^{\frac{p+3}{8}}}{2^{\frac{p-1}{4}}}
\pmod{p}.
\end{equation}
Combining (\ref{j2i2modp}), (\ref{g2jg2imodp}) and (\ref{2p14modp58}), we obtain that
$$
(-1)^{|\nu_g|}\equiv\prod_{1\leq i<j\leq \frac{p-1}{2}}\frac{g^{2j}-g^{2i}}{j^2-i^2}\equiv 
(-1)^{\frac{1}{4}(h(-p)+2)}\pmod{p}
$$
provided $p\equiv 5\pmod{8}$.
(\ref{tau58}) is concluded.

Next, assume that $p\equiv 1\pmod{8}$. Clearly now
$$
\frac{(p-3)(3p+1)}{16}+\frac12\equiv\frac{p-1}{8}\pmod{2}.
$$
According to (\ref{j2i2modp}), (\ref{g2jg2imodp}) and (\ref{2p14modp18}), we know that
\begin{equation}\label{taugp18}
(-1)^{|\tau_g|}\equiv\prod_{1\leq i<j\leq \frac{p-1}{2}}\frac{g^{2j}-g^{2i}}{j^2-i^2}\equiv 
(-1)^{\frac{1}{4}h(-p)+1}\bi\cdot\bigg(\frac{p-1}{2}\bigg)!\pmod{\fp}.
\end{equation}
On the other hand, by view of (\ref{kp1kmodp}), we have
\begin{align}\label{jip1modp}
\prod_{1\leq i<j\leq p-1}(j-i)=\prod_{i=1}^{p-2}(p-1-i)!=\prod_{k=1}^{p-2}k!\equiv(-1)^{\frac{p^2-9}{8}}\cdot\bigg(\frac{p-1}{2}\bigg)!\pmod{p}.
\end{align}
So combining (\ref{jip1modp}) with (\ref{petrov}), we obtain that
\begin{align}\label{sigmagp18}
(-1)^{|\sigma_g|}\equiv\prod_{1\leq i<j\leq p-1}\frac{g^j-g^i}{j-i}\equiv 
&\frac{(-1)^{\frac{p^2-9}{8}+\frac{p-1}2}\cdot e^{
\frac{(p-2)(3p-1)}{4}\cdot\pi\bi}}{(\frac{p-1}{2})!}\notag\\
\equiv &
(-1)^{\frac{p^2-9}{8}}\cdot e^{
\frac{(p-2)(3p-1)}{4}\cdot\pi\bi}\cdot\bigg(\frac{p-1}{2}\bigg)!\pmod{\fp}.
\end{align}
Note that
$$
\frac{p^2-9}{8}\equiv0\pmod{2}
,\qquad \frac{(p-2)(3p-1)}{4}\equiv\frac32\pmod{2}
$$
for any $p\equiv 1\pmod{8}$. Thus (\ref{tau18}) immediately follows from (\ref{taugp18}) and (\ref{sigmagp18}).

\section{The case $p\equiv 3\pmod{4}$}
\setcounter{lemma}{0}
\setcounter{theorem}{0}
\setcounter{corollary}{0}
\setcounter{remark}{0}
\setcounter{equation}{0}
\setcounter{conjecture}{0}

\begin{proof}[Proof of (iii) of Theorem \ref{main}]
Suppose that $p=18(2n+1)^2+1$ is a prime for some positive integer $n$. 
Let $\zeta$ and $\fp$ be the ones in (\ref{gzetafp}). Note that (\ref{Upsilonzeta}) is factly valid for each odd prime $p$, i.e.,
$$
\Upsilon(\zeta)=e^{\frac{(p-3)(3p+1)}{16}\cdot\pi\bi}\cdot\bigg(\frac{p-1}{2}\bigg)^{\frac{p-1}{4}}.
$$
Since $p=18(2n+1)^2+1\equiv 3\pmod{16}$ now, we have
\begin{equation}\label{g2jg2ip2n121}
\prod_{1\leq i<j\leq \frac{p-1}{2}}(g^{2j}-g^{2i})\equiv e^{\frac{(p-3)(3p+1)}{16}\cdot\pi\bi}\cdot\bigg(\frac{p-1}{2}\bigg)^{\frac{p-1}{4}}
=(6n+3)^{\frac{p-1}{2}}
\pmod{\fp}.
\end{equation}
Noting that both sides of (\ref{g2jg2ip2n121}) are rational integers, we factly get
\begin{equation}
\prod_{1\leq i<j\leq \frac{p-1}{2}}(g^{2j}-g^{2i})\equiv (6n+3)^{\frac{p-1}{2}}
\pmod{p}.
\end{equation}
On the other hand, in view of (\ref{ijp12j2i2p14}) and (\ref{kp1kmodp}),
\begin{align*}
\prod_{1\leq i<j\leq \frac{p-1}{2}}(j^2-i^2)
=\frac{1}{(\frac{p-1}{2})!}\prod_{k=0}^{\frac{p-3}{2}}(2k+1)!
=\prod_{k=0}^{\frac{p-5}{4}}(2k+1)!(p-2-2k)!
\equiv1\pmod{p}.
\end{align*}
Hence
$$
(-1)^{|\tau_g|}=\prod_{1\leq i<j\leq \frac{p-1}{2}}\frac{g^{2j}-g^{2i}}{j^2-i^2}\equiv 
(6n+3)^{\frac{p-1}{2}}\equiv\jacob{6n+3}{p}\pmod{p}.
$$
Since $p\equiv 3\pmod{4}$, by the law of quadratic reciprocity, we have
$$
\jacob{6n+3}{p}=(-1)^{3n+1}\cdot\jacob{p}{6n+3}=(-1)^{n+1}\cdot\jacob{2(6n+3)^2+1}{6n+3}=(-1)^{n+1}.
$$
The proof of (\ref{tau182n121}) is complete.
\end{proof}
\begin{proof}[Proof of (iv) of Theorem \ref{main}]
Suppose that $p\equiv 3\pmod 4$ is  a prime not of the form $18(2n+1)^2+1$. 
The case $p=3$ can be verified directly. So without loss of generality, below we assume that $p>3$. We mention that $(p-1)/2$ can't be a perfect square. Otherwise, assume on the contrary that $p=2m^2+1$ for some integer $m$. If $3\nmid m$, then $$
p=2m^2+1\equiv 2+1\equiv 0\pmod{3},
$$
which is impossible since $p>3$ is prime. Suppose that $3\mid m$. Since $p\equiv 3\pmod{4}$, $m$ must be odd. So $m=3(2n+1)$ for some integer $n$. This also contradicts with our assumption that $p$ isn't of the form $18(2n+1)^2+1$.

Let $$h=\frac{p-1}2$$ and let $\Z_h=\Z/h\Z$ denote the cyclic group of order $h$. For convenience, we  may write $\Z_h=\{1,2,\ldots,h\}$. On the other hand, recall that we have viewed $\cH_p=\{1,2,\ldots,h\}$ as a subset of $\F_p$. Let $\psi$ be the natural bijection from $\Z_h$ to $\cH_p$.

Assume that $1\leq a\leq p-1$ is prime to $p-1$. Define
$$
\eta_a(b):=ab
$$
for each $b\in\Z_h$. Since $a$ is also prime to $h$, $\eta_a$ is a permutation over $\Z_h$.
We claim that
\begin{equation}\label{taugataug}
\tau_{g^a}=\tau_{g}\circ\psi\circ\eta_a\circ\psi^{-1}
\end{equation}
for each $g\in\cR_p$.
In fact, assume that $1\leq b\leq h$ and $ab\equiv c\pmod{h}$ for some $1\leq c\leq h$. Clearly
$$
c=\psi\circ\eta_a\circ\psi^{-1}(b)
$$
provided that both $b$ and $c$ are viewed as the elements of $\cH_p$.
Note that either $ab\equiv c\pmod{p-1}$ or $ab\equiv c+h\pmod{p-1}$ now. Hence, we must have 
$$
g^c\equiv (g^a)^b\pmod{p}
$$
or 
$$
g^c\equiv (g^a)^b\cdot g^h\equiv -(g^{a})^b\pmod{p}.
$$
That is,
$$
\tau_{g^a}(b)=\tau_{g}(c).
$$

In view of (\ref{taugataug}), we get
$$
(-1)^{\tau_{g^a}}=(-1)^{\tau_{g}}\cdot(-1)^{|\eta_a|}.
$$
We shall find $1\leq a\leq p-1$ with $(a,p-1)=1$ such that $\eta_a$ is an odd permutation. 
Since $h$ is odd, according to the generalized Zolotarev lemma (cf. \cite{BC}), we know that
$$
(-1)^{|\eta_a|}=\jacob{a}{h},
$$
where $\jacob{\cdot}{\cdot}$ denotes the Jacobi symbol.
We may write $h=q_1^{\alpha_1}q_2^{\alpha_2}\cdots q_s^{\alpha_s}$, where $q_1,\ldots,q_s$ are distinct odd primes and $\alpha_1,\ldots,\alpha_s\geq 1$. Recall that $h$ can't be a perfect square. Without loss of generality, assume that $\alpha_1$ is odd. By the Chinese remainder theorem, there exists $1\leq a\leq p-1$ with $(a,p-1)=1$ such that $a$ is a quadratic non-residue modulo $q_1$ and is a quadratic residue modulo $q_i$ for each $2\leq i\leq s$. Then 
$$
\jacob{a}{h}=\prod_{j=1}^s\jacob{a}{q_j}^{\alpha_j}=(-1)^{\alpha_1}=-1,
$$
i.e., $(-1)^{|\eta_a|}=-1$.

Finally, since $a$ is prime to $p-1$, clearly $\nu_a:\,g\mapsto g^a$ is a permutation over $\cR_p$. So
$$
|\{g:\,(-1)^{|\tau_g|}=1\}|=
|\{g:\,(-1)^{|\tau_{g^a}|}=-1\}|\leq|\{g:\,(-1)^{|\tau_{g}|}=-1\}|,
$$
and
$$
|\{g:\,(-1)^{|\tau_g|}=-1\}|=
|\{g:\,(-1)^{|\tau_{g^a}|}=1\}|\leq|\{g:\,(-1)^{|\tau_{g}|}=1\}|.
$$
Hence we must have $|\{g\in\cR_p:\,(-1)^{|\tau_g|}=1\}|=|\{g\in\cR_p:\,(-1)^{|\tau_g|}=-1\}|$.

\begin{acknowledgment}
The authors thank Professor Zhi-Wei Sun for his very helpful comments. The first author also thanks Professors Henri Cohen and Will Jagy for informing him of the paper \cite{WC82}.
\end{acknowledgment}

\end{proof}
  
\end{document}